\documentclass[11pt,reqno]{amsart}

\usepackage{marginnote,amssymb,mathrsfs,latexsym,color,mathabx}
\usepackage[utf8]{inputenc}

\usepackage[margin=1.3in]{geometry}

\usepackage[colorlinks=true, linkcolor=blue]{hyperref}
\usepackage{mathtools}

\newcommand{\Dom}{\mathrm{Dom}}

\newcommand{\bbR}{{\mathbb R}}

\def\bbR{{\mathbb R}}

\newcommand{\N}{\mathbb{N}}

\newcommand{\R}{\mathbb{R}}

\def\e{\mathrm{e}}

\usepackage{changes}

\numberwithin{equation}{section}

\DeclareFontFamily{U}{mathx}{\hyphenchar\font45}
\DeclareFontShape{U}{mathx}{m}{n}{
      <5> <6> <7> <8> <9> <10>
      <10.95> <12> <14.4> <17.28> <20.74> <24.88>
      mathx10
      }{}
\DeclareSymbolFont{mathx}{U}{mathx}{m}{n}
\DeclareFontSubstitution{U}{mathx}{m}{n}
\DeclareMathAccent{\widecheck}{0}{mathx}{"71}
\DeclareMathAccent{\wideparen}{0}{mathx}{"75}

\makeatletter
\newcommand{\leqnomode}{\tagsleft@true}
\newcommand{\reqnomode}{\tagsleft@false}
\makeatother

\newcommand{\dd}{\mathrm{d}}

\def\CB{\color{black} }

\theoremstyle{theorem}
\newtheorem{theorem}{\sc \textbf{Theorem}}[section]  
\newtheorem{proposition}[theorem]{\sc \textbf{Proposition}}   
\newtheorem{corollary}[theorem]{\sc \textbf{Corollary}}        
\newtheorem{lemma}[theorem]{\sc \textbf{Lemma}}

\theoremstyle{remark}

\begin{document}

\title[Poincaré inequalities on Lie groups]{Local and nonlocal Poincaré inequalities\\ on Lie groups} 

 \author[T.\ Bruno]{Tommaso Bruno}
\address{Department of Mathematics: Analysis, Logic and Discrete Mathematics, Ghent University,
Krijgslaan 281, 9000 Ghent, Belgium}
\email{tommaso.bruno@ugent.be}

\author[M.\ M.\ Peloso]{Marco M.\ Peloso}
\address{Dipartimento di Matematica, 
Universit\`a degli Studi di Milano, 
Via C.\ Saldini 50,  
20133 Milano, Italy}
\email{marco.peloso@unimi.it}

\author[M.\ Vallarino]{Maria Vallarino}
\address{Dipartimento di Scienze Matematiche ``Giuseppe Luigi Lagrange'',
  Politecnico di Torino, Corso Duca degli Abruzzi 24, 10129 Torino,
  Italy - Dipartimento di Eccellenza 2018-2022}
\email{maria.vallarino@polito.it}

\keywords{Lie groups, Poincar\'e inequality}

\thanks{{\em Math Subject Classification} 26D10, 43A80}

\thanks{All authors are  partially supported by the GNAMPA 2020 project ``Fractional Laplacians and subLaplacians on Lie groups and trees" and are members of the Gruppo Nazionale per l'Analisi
  Matematica, la Probabilit\`a e le loro Applicazioni (GNAMPA) of the
  Istituto Nazionale di Alta Matematica (INdAM). T. Bruno acknowledges
  support by the Research Foundation--Flanders (FWO) through the
  postdoctoral grant 12ZW120N}

\begin{abstract} We prove a local $L^p$-Poincar\'e inequality, $1\leq p < \infty$, on noncompact Lie groups endowed with a sub-Riemannian structure. We show that the constant involved grows at most exponentially with respect to the radius of the ball, and that if the group is nondoubling, then its growth is indeed, in general, exponential. We also prove a nonlocal $L^2$-Poincar\'e inequality with respect to suitable finite measures on the group.

\end{abstract}

\maketitle

\section{Introduction}
The aim of this paper is to establish two forms of Poincaré inequality on noncompact connected Lie groups. On the one hand, we shall obtain the Lie group analogue of the classical inequality on $\R^d$
\begin{equation}\label{PoincareRd}
\| f -f_B\|_{L^p(B)} \leq C r \|\nabla f\|_{L^p(B)},
\end{equation}
where $p\in [1,\infty)$, $f\in C^\infty(\R^d)$, $B$ is a ball and $f_B$ is the average of $f$ on $B$. On the other hand, we shall consider a nonlocal $L^2$-version of such inequality, which takes the form
\begin{equation}\label{globPoincareRd}
\|f \|_{L^2(\R^d\!,\,\mu)} \leq C \|\nabla f\|_{L^2(\R^d\!,\,\mu)}
\end{equation}
for certain finite measures $\mu$ which are absolutely continuous with respect to the Lebesgue measure and whose densities satisfy a certain decay condition at infinity. One should think e.g.\ to the case when $\mu$ is a Gaussian measure.

Extensions of the classical Poincaré inequality~\eqref{PoincareRd} to non-Euclidean settings have been widely studied in the last decades. A thorough overview of the literature would go out of the scope of the present paper, so we refer the reader to the milestone~\cite{HK} and the references therein. For what concerns Lie groups, a Poincaré inequality on unimodular groups can be obtained by combining~\cite[\S 8.3]{SC} and~\cite[Theorem 9.7]{HK}. In this paper we prove that a Poincaré inequality holds also on nonunimodular Lie groups endowed with a relatively invariant measure, and we also describe the behaviour of the Poincaré constant in a quantitative way. We show that this grows at most exponentially with respect to the radius of the ball, and that if the group is nondoubling then such growth is, in general, exponential. More precisely, in a class of Lie groups including the real hyperbolic spaces as a subclass, we estimate from below the constant involved in the Poincar\'e inequality with a quantity which grows exponentially with respect to the radius of the ball. 

Nonlocal inequalities such as~\eqref{globPoincareRd} have been introduced more recently, and the decay of the densities involved is measured in terms of the Euclidean Laplacian $\Delta$. After its establishment on $\R^d$ in~\cite{BBCG}, they were extended to unimodular Lie groups of polynomial growth in~\cite{RussSire}, where a sum of squares subelliptic sub-Laplacian plays the role of $\Delta$. In this paper, we extend their method to the nondoubling regime, where the sub-Laplacian in addition has a drift term.

As a classical application of the local Poincar\'e inequality, we show the so-called local parabolic Harnack principle for the sub-Laplacian with drift. Another application of our inequality is given in~\cite{BC}  to the study of spectral properties of Schr\"odinger operators on Lie groups.

\section{Setting and Preliminaries}
Let $G$  be a noncompact connected Lie group with identity $e$. We
denote by $\rho$ a right Haar measure, by $\chi$ a continuous positive
character of $G$, and by $\mu_\chi$ the measure with density $\chi$
with respect to $\rho$. As the modular function on $G$, which we
denote by $\delta$, is such a character, $\mu_\delta$ is a left Haar
measure on $G$. We denote it by $\lambda$. Observe also that
$\mu_1=\rho$. 

Let $\mathbf{X}= \{ X_1,\dots, X_\ell\}$ be a family of left-invariant
linearly independent vector fields which satisfy H\"ormander's
condition. Let $d_C(\, \cdot\, ,\,  \cdot\, )$ be its associated
left-invariant Carnot--Carathéodory distance. We let $|x|=d_C(x,e)$,
and denote by $B_r$ the ball centred at $e$ of radius $r$. The measure
of $B_r$ with respect to $\rho$ will be denoted
by $V(r)=\rho(B_r)$; we recall that  $V(r)= \lambda(B_r)$. It is well known, cf.~\cite{Guiv,Varopoulos1}, that there exist $d\in
\N^*$ depending on $G$ and $\bf{X}$, and $C>0$, such that
\begin{equation}\label{pallepiccole}
C^{-1} r^d \leq V(r) \leq C r^d\qquad \forall r\in (0,1]\,, \\
\end{equation}
and $D_0,D>0$ depending only on $G$,
such that, for $r\geq 1$, either $C^{-1} r^D \le V(r)\le C r^D$   or 
\begin{equation}\label{pallegrandi}
C^{-1}\e^{D_0 r} \leq V(r)\leq C \e^{Dr}.
\end{equation}
In the former case, the group $G$ is said to be of polynomial
growth, while in the latter case of exponential
growth.  \CB

For any character $\chi$, one has (see~\cite{HMM})
\begin{equation}\label{supinf}
\sup_{B_r} \chi = \e^{c(\chi) r},  \qquad
\mbox{where} \quad   c(\chi) = \bigg(\sum_{j=1}^\ell |X_j\chi(e)|^2\bigg)^{1/2}.
\end{equation}

Given a ball $B$  with respect to $d_C$, we denote by $c_B$ its center and by $r_B$ its radius, and we write $B=B(c_B,r_B)$; we also set $2B=B(c_B,2r_B)$. Moreover, for $R>0$ let $\mathcal B_R$ be the family of
all balls of radius $\leq R$ and  
\begin{equation}\label{LDR}
D(R,\chi)=
 \sup_{B\in\mathcal B_R} \frac{\mu_\chi(2B)}{\mu_\chi(B)}
= \sup_{0<r\leq R}\frac{\mu_\chi(B_{2r})}{\mu_\chi(B_r)}, 
\end{equation}
where the latter equality holds since \CB  $\mu_\chi(B(c_B,r)) = (\chi
\delta^{-1})(c_B) \mu_\chi(B_r)$ for all  $r>0$ and $c_B\in G$.
By~\eqref{pallepiccole},~\eqref{pallegrandi} and~\eqref{supinf} there is $C>0$ such that either $C^{-1} \leq D(R,\chi)\leq C$, or
\[
D(R,\chi) \leq C \e^{(2D
  - D_0 + 3c(\chi))R},
\]
for all $R>0$. Actually, $(G, d_C, \mu_\chi)$ is doubling if and only if $\chi=1$ and $(G, d_C, \rho)$ is doubling, in particular $G$ is of polynomial growth. Indeed, if $\chi \neq 1$ then $\mu_\chi(B_r)$ grows exponentially with $r$. To show this, notice that if $\chi \neq 1$, then there is $x\in G$ with $\chi(x)>1$; let $N$ be the lowest integer such that $N\geq |x|$, and notice that $B_N x^n\subseteq B_{(n+1)N}$. If $r>N$ and $n$ is the largest integer such that $(n+1)N\leq [r]$, then
\[
\mu_\chi (B_r) \geq \mu_\chi ( B_N x^{n}) = \chi(x)^{n} \mu_\chi(B_N)\geq  \chi(x)^{[r]/N-2} \mu_\chi(B_N),
\]
whence the conclusion. Hence, the metric measure space $(G, d_C, \mu_\chi)$ is locally
doubling, but not doubling in general. In this setting we studied
various function spaces~\cite{BPTV,BPV1,BPV2} and the Sobolev and
Moser--Trudinger inequalities~\cite{BPV3}, and we refer the reader to these references for more details on these matters.

\section{The local Poincaré inequality on Lie groups}\label{seclocal}
In this section we prove the $L^p$-Poincar\'e inequality for smooth functions on $(G, d_C, \mu_\chi)$.
Given a ball $B$ and $f\in C^\infty(G)$, we  denote by
$f_B^\chi$ its average over $B$ with respect to $\mu_\chi$,
\[
f_B^\chi = \frac{1}{\mu_\chi(B)}\int_Bf \, \dd \mu_\chi,
\]
and we let $|\nabla f|^2 = \sum_{j=1}^\ell (X_j f)^2$. If $S$ is a set of variables, we denote by $C(S)$ a constant depending only on the elements of $S$.

\begin{theorem}\label{teo:LPBB}
There exist a constant $C=C(G,\mathbf{X})>0$ and a universal constant $\alpha>0$ such that, for all $p\in [1,\infty)$, $R>0$, all balls $B$ of radius $r\in (0,R]$ and $f\in C^\infty(G)$,
\begin{equation}\label{RLP}
\| f - f_B^\chi\|_{L^p(B,\mu_\chi)} \leq C\,e^{ \frac{1}{p} [2c(\chi) + c(\chi \delta^{-1})] R}\, D(R, \chi)^{\alpha}  \, r \,\||\nabla  f|\|_{L^p(B,\mu_\chi)}.
\end{equation}
\end{theorem}

Notice that the Poincaré constant grows at most exponentially with respect to the radius of the ball. The exponential term cannot, in general, be removed. After establishing the theorem, indeed, we show that when $G$ is the so called ``$ax+b$'' group and $\mu_\chi=\lambda$ is a left Haar measure,  the growth of the constant is indeed exponential.

\begin{proof}
Let  $p\in [1,\infty)$ be given. We shall prove that for every ball $B$ of radius $r >0$ and $f\in C^\infty(G)$
\begin{equation}\label{LP}
\int_B |f - f_B^\chi|^p\, \dd \mu_\chi  \leq 2^p\, \e^{c(\chi\delta^{-1})r}  \e^{2c(\chi) r} \, \frac{\mu_\chi(B_{2r})}{\mu_\chi(B_r)} \,r^p \,  \int_{2B}   |\nabla  f|^p \, \dd \mu_\chi.
\end{equation}
Once~\eqref{LP} is at disposal, the Poincaré inequality can be obtained by classical arguments, see e.g.~\cite[Theorem 9.7]{HK}. A careful inspection of~\cite[Section 5]{Jerison}, in particular, shows how a Whitney decomposition of $B$ brings to the constant given in the statement. We omit the details, which would be tedious and an almost verbatim repetition of the arguments that the reader can find in~\cite{Jerison}.

We then show~\eqref{LP}. For $z\in G$, let $\gamma_z \colon [0,1] \to G$ be such that $\gamma_z(0)=e$, $\gamma_z({|z|})=z$, $\gamma_z(s)\in B_{|z|}$ for every $s\in [0,|z|]$.

Let $B$ be a ball of radius $r>0$. Observe that if $x,y \in B$, and $z= x^{-1}y$, then $|z|<2r$. Moreover, if $\zeta = x \gamma_z(s)$, then $\zeta \in 2B$. For every $x,z\in G$, by H\"older's inequality
\begin{equation}\label{pointwisedifference}
|f(x) - f(xz)|^p\leq \bigg(\int_0^{|z|} |\nabla  f(x\gamma_z(s))|\, \dd s\bigg)^p \leq |z|^{p-1}\int_0^{|z|} |\nabla  f(x\gamma_z(s))|^p\, \dd s.
\end{equation}
We then have
\begin{align*}
\int_B |f - f_B^\chi|^p\, \dd \mu_\chi 
& = \int_B \left| \frac{1}{\mu_\chi(B)} \int_B \left( f(x)-f(y)\right)\, \dd \mu_\chi(y)\right|^p \, \dd \mu_\chi(x)\\
& \leq  \frac{1}{\mu_\chi(B)} \int_B \int_B \left| f(x)-f(y)\right|^p\, \dd \mu_\chi(y) \, \dd \mu_\chi(x),
\end{align*}
and after the change of variables $y=xz$, we get 
\begin{align*}
 \int_B |f - f_B^\chi|^p\, \dd \mu_\chi   \leq  \frac{1}{\mu_\chi(B)} \int_G \int_G \mathbf{1}_B(x)\mathbf{1}_B(xz)\left| f(x)-f(xz)\right|^p (\chi\delta^{-1})(x)\, \dd \mu_\chi(x)\, \dd \mu_\chi(z).
\end{align*}
Observe now that by~\eqref{pointwisedifference} and Fubini's theorem, we get
\begin{align*}
&\int_G \mathbf{1}_B(x)\mathbf{1}_B(xz)\left| f(x)-f(xz)\right|^p (\chi\delta^{-1})(x)\, \dd \mu_\chi(x)\\
& \qquad \leq  \frac{1}{\mu_\chi(B)} |z|^{p-1}  \int_0^{|z|}  \int_G \mathbf{1}_B(x)\mathbf{1}_B(xz)|\nabla  f(x\gamma_z(s))|^p\, (\chi\delta^{-1})(x)\,\dd \mu_\chi(x)\, \dd s .
\end{align*}
We make the change of variables $\zeta = x\gamma_z(s)$, and observe that if $x\in B$, then $(\chi\delta^{-1})(x) \leq (\chi\delta^{-1})(c_B) \, \sup_{B_r }(\chi\delta^{-1})$, and that $\chi(\gamma_z(s)) \leq \e^{2 c(\chi)r}$ by~\eqref{supinf}. Thus
\begin{align*}
&\int_G \mathbf{1}_B(x)\mathbf{1}_B(xz)  (\chi\delta^{-1})(x) |\nabla  f(x\gamma_z(s))|^p\, \dd \mu_\chi(x)\\
& \hspace{2cm} \leq \e^{2 c(\chi)r} (\chi\delta^{-1})(c_B) \, \sup_{B_r }(\chi\delta^{-1}) \, \mathbf{1}_{B_{2r}}(z) \int_{2B}   |\nabla  f(\zeta)|^p\, \dd \mu_\chi(\zeta).
\end{align*}
Therefore
\begin{align*}
\int_B |f - f_B^\chi|^p\, \dd \mu_\chi 
& \leq \e^{2 c(\chi)r} \e^{c(\chi\delta^{-1})r}  (2r)^p \, \frac{(\chi\delta^{-1})(c_B)}{\mu_\chi(B)}   \mu_\chi(B_{2r}) \int_{2B}   |\nabla  f(\zeta)|^p \, \dd \mu_\chi(\zeta),
\end{align*}
and~\eqref{LP} follows by \eqref{supinf}.
\end{proof}

As a corollary, we obtain the so-called local 
parabolic Harnack principle. We introduce the operator
\begin{equation}\label{Deltachi}
\Delta_{\chi} =-\sum_{j=1}^{\ell}(X_j^2 +(X_j\chi)(e)X_j ),
\end{equation}
which is essentially self-adjoint on $L^2(\mu_\chi)$ and
non-negative; see e.g.~\cite{HMM, BPTV}. We say that $\Delta_\chi$ satisfies the local parabolic Harnack
principle up to distance $R>0$ if there is $C(R)>0$ such that, for all $x\in G$, $r\in(0,R]$,
$s\in\bbR$, and any positive solutions $u$ of $(\partial_t +\Delta_\chi)u=0$ on
$(s,s+r^2)\times B(x,r)$, we have that
\begin{equation}\label{eq:PHP}
\sup_{Q_-} u \le C(R) \inf_{Q_+} u
\end{equation}
where 
\[
Q_- = \big( s+r^2/6, s + r^2/3\big) \times B(x,r/2),\qquad
Q_+ = \big( s+2r^2/3, s+r^2\big) \times B(x,r/2) .
\]

The following result follows at once from Theorem~\ref{teo:LPBB} and~\cite[Theorem 2.1]{SC}.
\begin{corollary}
For every $R>0$, $\Delta_\chi$ satisfies the local parabolic
  Harnack principle~\eqref{eq:PHP}. Furthermore, the positive 
  $\Delta_\chi$-harmonic functions satisfy the classical Harnack principle.
\end{corollary}

\subsection{Exponential growth of the constant}\label{sec:optimality}
We shall now show that if a Poincaré inequality 
\begin{equation}\label{Crp}
\int_{B_r} |f - f_B^\chi|^p\, \dd \mu_\chi  \leq C(r, p) \int_{B_r}   |\nabla  f|^p \, \dd \mu_\chi
\end{equation}
holds on general Lie groups for some $p$, then $C(r,p)$ grows exponentially with respect to the radius $r$. We show this on the $ax+b$ group of arbitrary dimension. {For notational convenience, we shall write $A \lesssim B$ to indicate that there is a constant $C$ such that $A \leq CB$. If $A\lesssim B$ and $B\lesssim A$, then we write $A\approx B$}.  

Let $G=\R^{n-1} \rtimes \R^+$ and let $(x,a)$ be its generic element. Recall that 
\[
\dd\lambda(x,a) = \frac{\dd x\, \dd a}{a^n} \qquad{\rm{and}}\qquad  \dd\rho(x,a) = \frac{\dd x\, \dd a}{a},
\]
since $\delta(x,a) = a^{-n+1}$; all positive characters of $G$ are of the form $ \chi_\gamma (x,a) = a^\gamma$ for some $\gamma \in \R$. We shall write $\mu_\gamma$ for the measure $\mu_{\chi_\gamma}$. {In particular,} $\lambda = \mu_{1-n}$ is the hyperbolic measure. We consider the left-invariant vector fields $X_i=a\partial_i$, $i=1,\dots,n-1$, and $X_0=a\partial_a$ which form a basis of the Lie algebra of $G$. The distance induced by such vector fields is the hyperbolic metric which is given by 
$$\cosh |(x,a)| = \frac{1}{2}(a+a^{-1} + a^{-1}|x|^2),$$ 
where $|x|$ is the Euclidean norm of $x\in \R^{n-1}$ (see \cite[(2.18)]{ADY}, \cite[(1.1)]{SV}). Then
\[
B_r = \left\{(x,a)\colon \e^{-r}<a<\e^r, \; |x|^2 < 2a(\cosh r - \cosh \log a)\right\}.
\]
In the case of the real hyperbolic space, i.e. the $ax+b$ group endowed with the measure $\lambda$ and the metric defined above, the constant $C(r,p)$ in \eqref{Crp} was estimated from above in \cite[Section 10.1]{HK}. We now estimate such constant from below. 

\smallskip

Consider the function $\phi \colon G\to \R$ defined by 
\[
\phi(x,a) = x_1,\qquad (x,a)\in G.
\]
Observe that $\int_{B_r} \phi \, \dd \mu_\gamma =0$ for all $\gamma\in\mathbb R$ and $|\nabla \phi (x,a) |^p = a^p$. Moreover,  
\[
\int_{B_r} |\phi|^p \, \dd \mu_\gamma\approx \int_{\e^{-r}}^{\e^r} a^{\gamma-1 + \frac{p+n-1}{2}} (\cosh r - \cosh \log a)^{\frac{p+n-1}{2}} {\dd a},
\]
while
\[
\int_{B_r} |\nabla \phi|^p \, \dd \mu_\gamma  \approx \int_{\e^{-r}}^{\e^r} a^{\gamma -1 +p+ \frac{n-1}{2}}{(\cosh r - \cosh \log a)^{\frac{n-1}{2}}} \, \dd a .
\]

\begin{lemma}
Let $\delta \in \R$ and $\epsilon >0$. Then
\[
\int_{\e^{-r}}^{\e^r} a^{\delta}{(\cosh r - \cosh \log a)^{\epsilon}} \, \dd a \approx \e^{r(|\delta +1| + \epsilon) }.
\]
\end{lemma}

\begin{proof}
We first make a change of variables
\[
\int_{\e^{-r}}^{\e^r} a^{\delta}{(\cosh r - \cosh \log a)^{\epsilon}} \, \dd a = \int_{-r}^r \e^{t(\delta +1)} (\cosh r - \cosh t)^{\epsilon} \, \dd t.
\]
Since $\cosh r - \cosh t \approx \e^r$ if $ |t|<r-1$, while $\cosh r - \cosh t \approx (r-|t|)\e^r$ if $ r-1<|t|<r$, we get
\[
\int_{-r+1}^{r-1} \e^{t(\delta +1)} (\cosh r - \cosh t)^{\epsilon} \, \dd t\approx \e^{\epsilon r} \int_{-r+1}^{r-1} \e^{t(\delta +1)}\, \dd t \approx \e^{(\epsilon + |\delta+1|) r} 
\]
while
\[
\int_{r-1<|t|<r} \e^{t(\delta +1)} (\cosh r - \cosh t)^{\epsilon} \, \dd t\approx \e^{(\epsilon +|\delta +1|) r} \int_{r-1<|t|<r}  (r-|t|)^\epsilon\, \dd t \approx \e^{(\epsilon + |\delta+1|) r} ,
\]
as required.
\end{proof}
From the lemma above, we get that
\[
\int_{B_r} |\phi|^p \, \dd \mu_\gamma\approx \e^{(  {|\gamma+ \frac{p+n-1}{2}|} + \frac{p+n-1}{2})r},
\]
while
\[
\int_{B_r} |\nabla \phi|^p \, \dd \mu_\gamma  \approx \e^{(  {|\gamma+p+ \frac{n-1}{2}| }+ \frac{n-1}{2})r}.
\]
We observe that, if $\gamma< {-\frac{p+n-1}{2}}$, then
\[
{\bigg|\gamma + \frac{p+n-1}{2}\bigg|} + \frac{p+n-1}{2} >  { \bigg|\gamma+p+ \frac{n-1}{2} \bigg|} + \frac{n-1}{2}.
\]
Thus for such $\gamma$
\[
C(r,p) \geq C  \e^{ r( {\left|\gamma + \frac{p+n-1}{2}\right | }+ \frac{p}{2}      -   { \left |\gamma +p+ \frac{n-1}{2} \right|}     )}.
\]
If in particular $\gamma= -n+1$, hence $\mu_\gamma$ is the left measure, and $n>p+1  $, then 
\[
 C(r,p) { \geq \e^{ (\left|   \frac{-n+p+1}{2}   \right|-  \left| \frac{-n+1}{2}+p  \right| +\frac{p}{2})r } =
 \begin{cases}
 \e^{pr}  &  n \geq 2p+1\\
 \e^{(n-p-1)r} & p+1 <n\leq 2p+1.
 \end{cases}
  }
\]

 \section{Nonlocal Poincaré inequality}\label{sec:5}
 In this second part of the paper we prove a nonlocal $L^2$-Poincar\'e
inequality for suitable finite measures on $G$ in the spirit of
\cite{MRS, RussSire}. More precisely, let $M$ be a positive function in
$L^1(\mu_\chi)$ and $\mu_{\chi,M}$ be the finite measure whose density
is $M$ with respect to $\mu_\chi$. We shall prove $L^2$-global Poincaré
inequalities for the measure $\mu_{\chi,M}$ for a large family of functions $M$. In order to do this, we let  
\[
L^2_1(\mu_{\chi,M}) = \{ f\in L^2(\mu_{\chi,M}) \colon |\nabla f| \in L^2(\mu_{\chi,M}) \}
\]
and introduce the operator
\begin{equation}\label{DeltaMchi}
\Delta_{\chi,M} = \Delta_\chi - \nabla(\log M) \cdot \nabla,
\end{equation}
where $\Delta_\chi$ is that of~\eqref{Deltachi}, $\Dom(\Delta_{\chi,M}) = \left\{f \in L^2_1(\mu_{\chi,M})\colon \Delta_{\chi,M} f \in L^2(\mu_{\chi,M}) \right\}$ and the derivatives are meant in the distributional sense. Observe that $\Delta_{\chi,M}$ is symmetric on $L^2(\mu_{\chi,M})$; in particular, for all $f\in \Dom(\Delta_{\chi,M})$ and $g\in L^2_1(\mu_{\chi,M})$,
\[
\int_{G} \nabla f \cdot \nabla g\, \dd\mu_{\chi,M}=\int_{G} \Delta_{\chi,M} f \cdot g\, \dd\mu_{\chi,M},
\] 
where $\nabla f \cdot \nabla g = \sum_{j=1}^\ell (X_j f)(X_j g)$. 

We say that the couple $(\Delta_\chi,M)$ admits a Lyapunov function if there exist a $C^2$ function $W\colon G\rightarrow [1,\infty)$ and constants $\theta>0$, $b\geq 0$, $R>0$ such that
\begin{equation} \label{lyap}
- \Delta_{\chi,M} W(x)\leq -\theta W(x)+b{\bf 1}_{B_R}(x) \qquad \forall x\in G.
\end{equation}
Observe that the existence of a Lyapunov function depends on $G$, $\mathbf{X}$, $\chi$ and $M$. For $f\in L^2_1(\mu_{\chi,M})$ we let
 \[
f_{\chi,M} = \frac{1}{\mu_{\chi,M}(G)}\int_G f\, \dd\mu_{\chi,M}.
\] 
Our  second main result  is the following global $L^2$-Poincar\'e inequality for  $\mu_{\chi,M}$.
\begin{theorem} \label{teo:GP}
If $(\Delta_\chi,M)$ admits a Lyapunov function, then there exists a constant $C=C(G, \mathbf{X}, \chi, M)$ such that for all $f\in L^2_1(\mu_{\chi,M})$
\begin{equation} \label{eqpoincmu}  
\| f- f_{\chi,M}\|_{L^2(\mu_{\chi,M})} \leq C \| |\nabla f|\|_{L^2(\mu_{\chi,M})}.
\end{equation}
\end{theorem}
Theorem \ref{teo:GP} is a generalization to any connnected noncompact possibly nonunimodular Lie group of the results proved in \cite{MRS} in the Euclidean setting and in~\cite{RussSire} in unimodular Lie groups of polynomial growth, by which our proof is inspired.  More general versions of nonlocal Poincar\'e inequalities of this kind were proved in \cite{Gr} in the setting of a topological measure space endowed with a family of sets which play the role of unit balls and satisfy suitable assumptions. Recently in \cite{CFZ} nonlocal $L^p$-Poincar\'e inequalities were obtained on Carnot groups of Engel type in the  case when the density of the measure depends on a homogeneous norm of the group; we note however that the case $p=2$ is always excluded.

\begin{proof}
Let $f\in L^2_1(\mu_{\chi,M})$, and observe first that
\begin{equation}\label{inf}
\int_G \left| f- f_{\chi,M} \right|^2 \dd \mu_{\chi,M} = \min_{c\in \R} \int_G \left| f- c\right|^2 \dd \mu_{\chi,M}.
\end{equation}
Let now $g=f-c$ for a positive $c$ to be determined, and $W$ be a Lyapunov function for $(\Delta_\chi,M)$. By~\eqref{lyap} 
\begin{equation} \label{splitting}
\int_G |g|^2 \, \dd\mu_{\chi,M} \leq \int_G |g|^2 \frac{\Delta_{\chi,M} W}{\theta W}\, \dd \mu_{\chi,M} +\int_{B_R} |g|^2 \frac{b}{\theta W}\, \dd\mu_{\chi,M}.
\end{equation}
We treat the two terms separately.

Let us consider the first term, and prove that
\begin{equation}\label{firsterm}
\int_G \frac{\Delta_{\chi,M} W}{W}g^2 \dd \mu_{\chi,M} \leq \int_G |\nabla g |^2\dd\mu_{\chi,M}. 
\end{equation}
We prove it by density, and first assume that $g$ is compactly supported.  By definition of $\Delta_{\chi,M}$,
\begin{align*}
\int_G \frac{\Delta_{\chi,M} W}{W} g^2 \,\dd \mu_{\chi,M}
&= \int_G \nabla \left(\frac{g^2}W\right)\cdot \nabla W\, \dd\mu_{\chi,M} \\
& =   2 \int_G \frac{g}{W} \nabla g\cdot \nabla W\, \dd\mu_{\chi,M} -\int_G \frac{g^2}{W^2} |\nabla W|^2 \, \dd\mu_{\chi,M}\\
& = \int_G |\nabla g|^2\, \dd\mu_{\chi,M}  - \int_G |\nabla g-\frac{g}{W}  \nabla W|^2 \, \dd\mu_{\chi,M} \\
& \leq  \int_G |\nabla g|^2\dd\mu_{\chi,M}.
\end{align*}
Let now $g\in L^2_1(\mu_{\chi,M})$,  and consider a nondecreasing sequence of functions $\psi_n\in C_c^\infty(G)$ such that
\[
{\bf 1}_{B_{nR}} \leq \psi_n\leq 1, \qquad |\nabla \psi_n| \leq 1.
\]
By applying ~\eqref{firsterm} to $g\psi_n$, the monotone convergence theorem in the left-hand side and the dominated convergence theorem in the right-hand side, one gets~\eqref{firsterm}.

To deal with the second term, we choose $c$ such that $\int_{B_R}g\, \dd\mu_\chi=0$. By~\eqref{LP} applied to $g$ on $B_R$, and the fact that $M$ is bounded from above and below on $B_R$, one has
\begin{align*}
\int_{B_R} |g|^2 \, \dd\mu_{\chi,M} \leq C \int_{B_R} |g|^2 \, \dd\mu_\chi \leq C \int_{B_{2R}} |\nabla g|^2\, \dd\mu_\chi \leq C \int_{B_{2R}} |\nabla g|^2\, \dd\mu_{\chi,M}
\end{align*}
where the constant $C$ depends on $R$ and $M$.  Therefore, since $W\geq 1$,
\begin{equation*}
\int_{B_R} |g|^2 \frac{b}{\theta W}\, \dd\mu_{\chi,M} \leq C \int_{B_{2R}} |\nabla g|^2\, \dd\mu_{\chi,M} \leq C \int_{G} |\nabla g|^2\, \dd\mu_{\chi,M} ,
\end{equation*}
which completes the proof.
\end{proof}

\begin{corollary} \label{corLyapunov}
Let $v= -\log M$. If there exist $a\in (0,1)$, $c>0$ and $R>0$ such that
\begin{equation}\label{suffLyap}
a|\nabla v|^2(x) + \Delta_\chi v(x)\geq c \qquad \forall \, x \in B_R^c,
\end{equation}
then $(\Delta_\chi,M)$ admits a Lyapunov function, and~\eqref{eqpoincmu} holds.
\end{corollary}

\begin{proof}
Let $W(x)= \e^{(1-a)\left(v(x)-\inf_Gv\right)}$, so that
\[
-\Delta_{\chi,M} W =(1-a)W \left(- \Delta_\chi v-a|\nabla v|^2\right).
\]
Then $W$ is a Lyapunov function with $\theta=c (1-a)$ and $b= \max_{B_R} (-\Delta_{\chi,M} W+\theta W)$.
\end{proof}

One can actually show that if~\eqref{suffLyap} holds with $a<1/2$, then~\eqref{eqpoincmu} self-improves as follows.
\begin{proposition} \label{selfimprGP} 
Let $v= -\log M$.  If there exist $c>0$, $R>0$ and $\epsilon\in (0,1)$ such that
\begin{equation} \label{selfimprovcond}
\frac{1-\epsilon}{2}|\nabla v|^2(x) + \Delta_\chi v(x)\geq c \qquad \forall \, x\in B_R^c,
\end{equation}
then there exists $C>0$ such that for all $f\in L^2_1(\mu_{\chi,M})$
\begin{equation} \label{pim} 
\| |f - f_{\chi,M}|\left(1+| \nabla v|\right)\|_{L^2(\mu_{\chi,M})} \leq C \| |\nabla f|\|_{L^2(\mu_{\chi,M})}.
\end{equation}
\end{proposition}
\begin{proof}
Observe first that, since $v$ is $C^2$ and \eqref{selfimprovcond} holds,
\begin{equation} \label{alpha}
\frac{1-\epsilon}2|\nabla v|^2 - \Delta_\chi v \geq \alpha
\end{equation}
for some $\alpha\in \R$. Let $f\in L^2_1(\mu_{\chi,M})$ and let $g=f\sqrt{M}$. Since
\[
\nabla f=\frac{1}{\sqrt{M}} \nabla g+\frac{1}{2}g \frac{1}{\sqrt{M}} \nabla v,
\]
by~\eqref{alpha} 
\begin{align*}
\int_{G}|\nabla f|^2 \, \dd \mu_{\chi,M}
&= \int_{G} \left(|\nabla g|^2  +\frac{1}{4}|g|^2|\nabla v|^2+ g \nabla g\cdot \nabla v\right) \, \dd \mu_\chi\\
&=   \int_{G} \left(|\nabla g|^2  +\frac{1}{4} |g|^2|\nabla v|^2+\frac{1}{2}\nabla(|g|^2)\cdot \nabla v\right) \, \dd \mu_\chi \\
&\geq  \int_{G} |g|^2 \left(\frac{1}{4}| \nabla v|^2+ \frac{1}{2} \Delta_\chi v \right) \, \dd \mu_\chi\\
&\geq {\frac{1}{2} \int_{G} |f|^2 \left( \frac{\epsilon}{2}|\nabla v|^2 + \alpha \right)\, \dd \mu_{\chi,M}.}
\end{align*}
Since~\eqref{eqpoincmu} holds by~\eqref{selfimprovcond} and Corollary~\ref{corLyapunov}, the conclusion follows.
\end{proof}


\begin{thebibliography}{9}

\bibitem{ADY} J.-Ph. \ Anker, E. \ Damek, C.\ Yacoub,
\emph{Spherical analysis on harmonic $AN$ groups}. 
Annali Scuola Norm. Sup. Pisa 33 (1996), 643--679.
  
\bibitem{BBCG}
D.\ Bakry, F.\ Barthe, P.\ Cattiaux, A.\ Guillin, \emph{A simple proof of the Poincaré inequality for a large class of probability measures including the log-concave case}, Electron.\ Comm.\ Probab.\ 13 (2008), 60--66.  
  
\bibitem{BC}
T.\ Bruno, M.\ Calzi, \emph{Schr\"odinger operators on Lie groups with purely discrete spectrum}, preprint.  
  
\bibitem{BPTV}
T.\ Bruno, M.\ M.\ Peloso, A.\ Tabacco, M.\ Vallarino, \emph{Sobolev spaces on Lie groups: embedding theorems and algebra properties}, J.\ Funct.\ Anal.\ 276 (2019), no.\ 10, 3014--3050.

\bibitem{BPV1} T.\ Bruno, M.\ M.\ Peloso, M.\ Vallarino, \emph{Besov and Triebel--Lizorkin spaces on Lie groups}. Math. Ann. 377 (2020), no. 1-2, 335--377.

\bibitem{BPV2} T.\ Bruno, M.\ M.\ Peloso, M.\ Vallarino, \emph{Potential spaces on Lie groups}. To appear on ``Geometric aspects of harmonic analysis'', Eds.\ P.\ Ciatti, A.\ Martini, Springer INdAM Series, (2021). arXiv:1903.06415   

\bibitem{BPV3} T.\ Bruno, M.\ M.\ Peloso, M.\ Vallarino, \emph{The Sobolev embedding constant on Lie groups}.  arXiv:2006.07056

\bibitem{CFZ} M.\ Chatzakou, S.\ Federico, B.\ Zegarlinski, \emph{q-Poincar{\`e} inequalities on Carnot Groups with a filiform Lie algebra}. arXiv:2007.04689   
 
  
\bibitem{Gr} P.T.\ Gressman, \emph{Fractional Poincar\'e and logarithmic Sobolev inequalities for measure spaces}, J. Funct. Anal. 265 (2013), no. 6, 867--889
 
\bibitem{Guiv} 
Y.\ Guivarc'h, \emph{Croissance polynomiale et périodes des fonctions harmoniques}, Bull.\ Soc.\ Math.\ France 101 (1973), 333--379.  


\bibitem{HK} P.\ Hajlasz, P.\ Koskela, \emph{Sobolev met Poincar\'e}. Mem. Amer. Math. Soc. 145 (2000), no. 688, x+101 pp.

\bibitem{HMM}W.\ Hebisch, G.\ Mauceri, S.\ Meda, \emph{Spectral multipliers for Sub-Laplacians with drift on Lie groups}, Math. Z.\ 251 (2005), no.\ 4, 899--927.

  

\bibitem{Jerison}
D.\ Jerison, \emph{The Poincaré inequality for vector fields satisfying H\"ormander's condition}, Duke Math.\ J.\ 53 (1986), no.\ 2, 503–523.



\bibitem{MRS} C.\ Mouhot, E. \ Russ, Y. \ Sire, \emph{Fractional Poincar\'e inequalities for general measures}. J. Math. Pures Appl. (9) 95 (2011), no. 1, 72--84.

\bibitem{RussSire}
E.\  Russ, Y.\ Sire, \emph{Nonlocal Poincaré inequalities on Lie groups with polynomial volume growth and Riemannian manifolds}. Studia Math.\ 203 (2011), no.\ 2, 105--127.

 
 \bibitem{SC} L. \ Saloff-Coste, \emph{ Parabolic Harnack inequality for divergence-form second-order differential operators. Potential theory and degenerate partial differential operators.}  Potential Anal. 4 (1995), no. 4, 429--467.
 
 \bibitem{SV} P.\ Sj\"ogren, M.\ Vallarino, \emph{Boundedness from $H^1$ to $L^1$ of Riesz transforms on a Lie group of exponential growth}. Ann. Inst. Fourier, Grenoble \ 58 (2008), no.4, 1117-1151. 
  
 
\bibitem{Varopoulos0} N.\ Th.\ Varopoulos, \emph{Fonctions harmoniques sur les groupes de Lie}. C. R. Acad. Sci. Paris, S\'erie
I, Math. 309, 1987, 519--521.
 
 \bibitem{Varopoulos1} 
N.\ Th.\ Varopoulos, \emph{ Analysis on Lie groups}, J.\ Funct.\ Anal.\ 76 (1988), no.\ 2, 346--410. 

  
 \end{thebibliography}
\end{document}